\documentclass[twoside,11pt]{article}
\usepackage{graphicx,amscd,amsfonts,amsmath,amsthm,amssymb,latexsym,amsfonts,color}
\usepackage{epsfig,float,epstopdf,array,bm,cite}
\usepackage[bookmarksnumbered, colorlinks,plainpages]{hyperref}

\usepackage{tabu}
\usepackage{colortbl}
\usepackage{multirow}
\usepackage{rotating}
\usepackage{hyperref}
\usepackage{amsfonts,algorithm,algorithmic}

\hypersetup{
	colorlinks=true,
	linkcolor=blue, anchorcolor=green, citecolor=magenta, urlcolor=cyan, filecolor=magenta
}
\usepackage{amsthm,amsmath}

\newcommand{\norm}[1]{\left\|{#1}\right\|}

\def\EPSS{\rm{EPSS}}
\def\SEPSS{\rm{SEPSS}}

\newcommand{\navierSEPSS}[3]{&{#1}&{#2}&{#3}}

\newcommand{\navierOTHER}[9]{&{#8}&{#2}&{#5}&{#9}&{#3}&{#6}&{#7}&{#1}&{#4}}

\def\EPSS{\rm{EPSS}}
\def\SEPSS{\rm{SEPSS}}
\def\trace{\rm{trace}}

\def\ind{\rm{index}}

\def\clnavier{\cline{2-15}}

\newtheorem{theorem}{\bf Theorem}
\newtheorem{lemma}{\bf Lemma}
\newtheorem{corollary}{\bf Corollary}

\newtheorem{definition}{\bf Definition}

\begin{document}
\title{\bf Semi-convergence of the EPSS method for singular generalized saddle point problems}
\author{\small\bf Mohsen Masoudi
                            ~and \small\bf Davod Khojasteh Salkuyeh\thanks{Corresponding author}\\  [2mm]
{\small \textit{Faculty of Mathematical Sciences, University of Guilan, Rasht, Iran }}\\
{\small \textit{Emails:  masoudi\_mohsen@phd.guilan.ac.ir,  khojasteh@guilan.ac.ir}}\\[1mm]
}
\date{}
\maketitle
\vspace{-0.5cm}

\noindent\rule{6.5in}{0.02in}\\
{\bf Abstract.}  Recently, in (M. Masoudi, D.K. Salkuyeh, An extension of positive-definite and skew-Hermitian splitting method for preconditioning of generalized saddle point problems, Computers \& Mathematics with Application, https://doi.org/10.1016/j.camwa.2019.10.030, 2019) an extension of the positive definite and skew-Hermitian splitting (EPSS) iteration method for nonsingular generalized saddle point problems has been presented. In this article, we study semi-convergence of the EPSS method for singular generalized saddle problems. 
Then a special case of EPSS (SEPSS) preconditioner is applied to the nonsingular generalized saddle point problems. Some numerical results are presented to show the effectiveness of the preconditioner. 
\\[-3mm]

\noindent\rule{6.5in}{0.02in}

\pagestyle{myheadings}\markboth{M. Masoudi, D.K. Salkuyeh}{ On the semi-convergence of the EPSS}
\thispagestyle{empty}

\section{Introduction} 
Consider the  generalized saddle point problems  of the form
\begin{equation}
\label{saddle}
\mathcal{A}u= 
\begin{bmatrix}{A}&{B^*}\\{-B}&{C}\end{bmatrix}
\begin{bmatrix}{x}\\{y}\end{bmatrix}= 
\begin{bmatrix}{f}\\{g}\end{bmatrix}
\equiv b,
\end{equation}
where $A \in \mathbb{C}^{n\times n}$ is non-Hermitian positive definite,  i.e, $A+A^*$ is Hermitian positive definite, and $C \in \mathbb{C}^{m\times m}$ is positive semi-definite,  i.e,  $C+C^*$ is Hermitian positive semi-definite. Moreover  $f \in \mathbb {C}^n$,  $g \in \mathbb{C}^m$ and $ B \in  \mathbb{C}^{m \times n} (m\leq n)$  is rank deficient.
We also assume that the matrices $A$, $B$ and $C$ are large and sparse.
The  generalized saddle point problem  can be found in a variety of scientific and engineering problems;
e.g., computational fluid dynamics, constrained optimization, etc.  (see \cite{bai752006,benzi142005}).

Since the nonsingular coefficient matrix of \eqref{saddle} has some good properties, one approach in the literature is to drop
some elements from matrix $B$ in order to eliminate its singularity (see \cite{benzi142005,simoncini262004}) and the other approach is to employ some special techniques in the modelling process so the resulted problem is directly nonsingular. For example, in the field of
electric networks, the rank deficiency in $B$ can be removed by grounding one of the nodes, however, after this operation, the
resulting linear systems maybe rather ill-conditioned, see \cite{bochev472005} for details. In addition, Cao \cite{cao1742006} has compared the convergence
performance of Krylov subspace methods for solving singular saddle point problem \eqref{saddle} and the corresponding nonsingular
saddle point problem by some numerical experiments, and found that the convergence behavior of the singular case is
significantly better than that of the corresponding nonsingular case and the reason why it has such difference is still an open
problem. Therefore, we will not transfer some singular cases into nonsingular cases directly. 

For nonsingular generalized saddle point problem, a number of iteration methods have been developed in the literature, such as the SOR-like method \cite{golub412001},  {the GSOR method \cite{bai1022005}},
the Uzawa method \cite{bramble341997},
the parametrized inexact Uzawa methods \cite{bai4282008},
the Hermitian and skew-Hermitian splitting
(HSS)   methods \cite{benzi262004, 	simoncini262004} and so on.
Also,  when the generalized saddle point problem \eqref{saddle} is singular, some iteration methods
and preconditioning techniques have been presented, 
such as the generalized parametrized inexact Uzawa methods \cite{yang692015}, 
HSS  method \cite{bai892010}, 
regularized HSS  method \cite{chao762018}, 
preconditioned deteriorated PSS  method \cite{liang782018}, 
relaxed deteriorated PSS preconditioner \cite{liang3052017}, 
GSS  method and preconditioner \cite{cao712016, huang3282018, chen2692015, shen722016}, 
inexact version of the GSS  method  \cite{huang372018}, modification of the GSS method \cite{salkuyeh742017}, 
constraint preconditioning  method \cite{yang2822015,zhang2382013} and so on. 

The PSS iteration method was presented by Bai et al. in \cite{bai262005} for the solution of non-Hermitian positive-definite linear systems. The PSS preconditioner is generated by the PSS iteration method.
In \cite{bai262005} it was  shown that the PSS iteration method converges unconditionally to the unique solution of the non-Hermitian system of linear equations. A generalization of the PSS iteration method was presented by Cao et al. in \cite{cao22012}. In \cite{pan1722006}, Pan et al. proposed the deteriorated positive-definite and skew-Hermitian splitting (DPSS) preconditioner for saddle point problems with $(2,2)$-block being zero. Then, Shen in \cite{shen2372014} applied the method to the generalized saddle point problems. Fan and Zhu in \cite{fan2582015} proposed  a  generalized relaxed positive-definite and skew-Hermitian splitting preconditioner for non-Hermitian saddle point problems. Xie and Ma in \cite{xie722016} presented the modified PSS preconditioner for generalized saddle point problems.

In this article, we study semi-convergence of the extended PSS (EPSS)  iteration method \cite{mas2019} for solving singular generalized saddle point problems. We show that the EPSS method is unconditionally semi-convergent for the singular generalized saddle point problems \eqref{saddle}. Since the EPSS iteration method is a general case of some other methods, we conclude that all of these iteration methods are semi-convergent. We also study the EPSS induced preconditioner.

{The following notations are used throughout this paper. The set of all
	$n \times n$ complex matrices and $n \times 1 $ complex  vectors are denoted by $\mathbb{C}^{n\times n} $ and $\mathbb{C}^n$, respectively.	
	The symbol $I$ denotes the identity matrix. 
	Notation ${A^*}$ is used for   the conjugate transpose of the matrix $A$.
	For a given square matrix $A$, $\sigma (A)$ stands for the spectrum of the matrix $A $.
	We denote the  spectral radius  of the matrix $A$ by $\rho(A)$ which is  defined by  $\rho(A)=\max \{|\lambda|: \lambda \in \sigma(A)\}.$
	The notation $null(A)$ is used to represent the null space of  the matrix $A$.
	We say that the matrix $A\in \mathbb{C}^{n\times n}$ is Hermitian positive-definite (HPD), if $A^*=A$ and  $x^* A x>0$  for all nonzero  $x \in \mathbb{C}^n$. Similarly, the matrix $A\in \mathbb{C}^{n\times n}$ is called  Hermitian positive semidefinite (HPSD), if $A^*=A$ and $x^* A x\geqslant 0$ for all  $x \in \mathbb{C}^n$.
	The matrix $A\in \mathbb{C}^{n\times n}$ is said to be positive-definite (positive semidefinite) if $A+A^*$ is HPD (HPSD).
	Throughout the paper we use PD (PSD)  to denote positive-definite (positive semidefinite) matrices.   
	If  $S^*=-S$, we say that $S$ is  a  skew-Hermitian matrix. For a complex number $z$, the real part and imaginary part of $z$ are denoted by $\Re(z)$ and $\Im(z)$, respectively. For two vectors $u$ and $v$, we use $[u;v]$ to denote $[u^T,v^T]^T$.
}

\section{Semi-convergence of EPSS iteration method}
\label{sectionepsssad}

We first give a brief description of the EPSS method \cite{mas2019}  for the generalized saddle point problems. 
Let  $A=A_P+A_S$ be a positive definite and skew-Hermitian  splitting of the matrix $A$, $C=C_P+C_S$ be  positive semi-definite  and skew-Hermitian  splitting of the matrix $C$ and   $B=B_P+B_S$ is an arbitrary splitting of the matrix $B$.
Let 
\begin{align}
\label{P,S}
{\mathcal P}=
\begin{bmatrix}{A_P}&{B_P^*}\\{-B_P}&{C_P} \end{bmatrix}, ~~
{\mathcal S}
=
\begin{bmatrix}{A_S}&{B_S^*}\\{-B_S}&{C_S}\end{bmatrix}, ~~
\Sigma=
\begin{bmatrix}{P_\alpha}&{0}\\{0}&{P_\beta}\end{bmatrix}.
\end{align}
where  $P_\alpha$ and $P_\beta$  are HPD matrices. Obviously we have  $\mathcal{A}= {\mathcal P}+ {\mathcal S}$. Using the shift matrix $\Sigma$ which is HPD and the splittings
\[
\mathcal{A} =  (\Sigma+ {\mathcal P}  ) -  (\Sigma - \mathcal S  )= ( \Sigma +  \mathcal S ) -  (\Sigma - {\mathcal P} ).
\]
we establish the EPSS iteration method as
\begin{align}
\label{seq}
\begin{cases}
(\Sigma+{\mathcal P}  ) {u}^{k+\frac{1}{2}}= (\Sigma - \mathcal S  )  {u}^{k}+{b},\\
(\Sigma +\mathcal S  ) {u}^{k+1}= ( \Sigma -{\mathcal P}  )  {u}^{k+\frac{1}{2}}+{b},	
\end{cases}
\end{align}
where $\mathbf{u}^0\in \mathbb{C}^n$ is a given initial guess. Eliminating ${u}^{k+\frac{1}{2}}$ from \eqref{seq} yields
\begin{align}
\label{EPSS}
{u}^{k+1}= {\Gamma}_{\EPSS} {u}^k+{c},
\end{align}
where
\begin{align}
\label{gama}
{\Gamma}_{\EPSS}= ( \Sigma +\mathcal  S  )^{-1}  ( \Sigma -{\mathcal P}   ) ( \Sigma + {\mathcal P}   )^{-1}  ( \Sigma  - \mathcal  S  ),
\end{align}
and ${c}=2  ( \Sigma + \mathcal  S  )^{-1} \Sigma   ( \Sigma + {\mathcal P}    )^{-1} \mathbf{b}$.
If we define
\begin{align}
\label{MN}
{ \mathcal{M}}= \frac{1}{2}  ( \Sigma +{\mathcal P}  )\Sigma ^{-1} ( \Sigma + \mathcal  S  )\quad \textrm{and} \quad
{\mathcal N}=\frac{1}{2}  (\Sigma -{\mathcal P}  )\Sigma ^{-1} ( \Sigma -\mathcal S  ),
\end{align}
then the matrix $\mathcal{M}$ is nonsingular and ${A}={\mathcal{M}}-{\mathcal{N}}$. Moreover, we have $\Gamma_{\EPSS}={\mathcal{M}}^{-1} {\mathcal{N}}=I-{\mathcal{M}}^{-1}\mathcal{A}$.

Therefore, if we use a Krylov subspace method such as GMRES or its restarted version \cite{saad71986} to approximate the solution of the system \eqref{saddle}, then the matrix $\mathcal{M}$ can be considered as a preconditioner to this system.
Since the  prefactor  $\frac{1}{2}$ in the preconditioner $ {\mathcal{M}}$  has no effect on the preconditioned system,  we can take the matrix
{\small
\begin{eqnarray}
\label{orgP}
\mathfrak{P}_{\EPSS}&=&\left( \Sigma +  {\mathcal P} \right)\Sigma ^{-1}\left( \Sigma + S \right)\nonumber\\
&=& 
\begin{bmatrix}{P_\alpha +A_P }&{B_P^*}\\{-B_P}&{P_\beta +C_P} \end{bmatrix}
\begin{bmatrix}{P_\alpha^{-1}}&{0}\\{0}&{P_\beta^{-1}} \end{bmatrix}
\begin{bmatrix}{P_\alpha +A_S  }&{B_S^*}\\{-B_S}&{P_\beta+C_S  }\end{bmatrix}
\nonumber\\
&=& 
\begin{bmatrix}
P_\alpha +A+A_PP_\alpha^{-1}A_S-B_P^*P_\beta ^{-1} B_S& (P_\alpha +A_P) P_\alpha^{-1} B_S^*+B^*_PP_\beta^{-1}(P_\beta+C_S)\\
-B_PP_\alpha^{-1}(P_\alpha +A_S)- (P_\beta +C_P)P_\beta^{-1}B_S  &P_\beta +C+C_P P_\beta^{-1}C_S-B_PP_\alpha ^{-1} B_S^*
\end{bmatrix}.
\end{eqnarray}}

The  convergence of the  EPSS method  for     nonsingular generalized saddle point problems  \eqref{saddle}  was studied, in \cite{mas2019}. 
In the following,   we study the semi-convergence of  this method for singular generalized saddle point problems. 

\begin{definition}
	Let $\mathcal{A}$ be singular. The iteration method \eqref{EPSS}  is called semi-convergent, if it  converges to a solution of the system $\mathcal{A} u=b$ for any initial guess $u^{0}$.
\end{definition}

\begin{lemma}  \cite{berman91994}
	The the  iteration method \eqref{EPSS} is semi-convergent, if and only if 
	\begin{enumerate}
		\item The elementary divisors of the iteration matrix $\Gamma$  associated with  $\lambda=1$ are linear, i.e., $ \ind(I-\Gamma ) = 1$, or equivalently, ${\rm rank}(I -\Gamma )^2 = {\rm rank}(I-\Gamma )$;  
		\item	The pseudo-spectral radius satisfies $\nu(\Gamma) <1$, where $ \nu(\Gamma)=\max \{|\lambda|: \lambda \in \sigma(\Gamma), \lambda\neq 1\}.$
	\end{enumerate}   
\end{lemma}

\begin{theorem}
	\label{semiconvth}
	Suppose that $A$ and $C$ are  PD and PSD, respectively. Let  $\lambda \in \sigma(\Gamma_{\EPSS})$ with $|\lambda|=1$.
	If   for all $0\neq r \in null (C+C^*)$, we have 
	\begin{align}
	\label{condi2}
	r^*(P_\beta-B_S P_\alpha^{-1} B_P^*+C_SP_\beta^{-1} C_P)r\neq\frac{1+\lambda}{1-\lambda}r^*Cr,
	\end{align}
	then the EPSS method is   semi-convergent for any initial guess $x^ 0 $.    
\end{theorem}	
\begin{proof}  
	At first, we show    the elementary divisors of the iteration matrix $\Gamma_{EPSS}=I-\mathfrak {P}_{EPSS}^{-1} \mathcal{A}$  associated with  $\lambda=1$ are linear.
	To do so, it is enough to show that  $null\left((\mathfrak {P}_{EPSS}^{-1}\mathcal{A})^2\right)= null(\mathfrak {P}_{EPSS}^{-1}\mathcal{A})$ (see \cite{salkuyeh742017}).
	Since $ null(\mathfrak{P}_{EPSS}^{-1}\mathcal{A})\subseteq null\left((\mathfrak{P}_{EPSS}^{-1}\mathcal{A})^2\right)$, 
	it is sufficient to prove   
	\begin{align}
	\label{null}
	null\left((\mathfrak{P}_{EPSS}^{-1}\mathcal{A})^2\right)\subseteq null(\mathfrak{P}_{EPSS}^{-1}\mathcal{A}).
	\end{align} 
	Let $q=[q_1;q_2] \in null\left((\mathfrak{P}_{EPSS}^{-1}\mathcal{A})^2\right)$. 
	Therefore $(\mathcal{P}_{EPSS}^{-1}\mathcal{A})^2 q=0$. 
	Denote by $r=[r_1;r_2]=\mathfrak{P}_{EPSS}^{-1}\mathcal{A} q$. 
	Hence $\mathfrak{P}_{EPSS}^{-1}\mathcal{A}r=0$. To show \eqref{null}, we prove $r_1=0$ and $r_2=0$. 
	Since  $\mathfrak{P}_{EPSS}^{-1}\mathcal{A}r=0$, we have $\mathcal{A}r=0$. 
	Therefore 
	\begin{align}
	\label{aq}
	\begin{cases}
	Ar_1+B^*r_2=0,\\
	-Br_1+Cr_2=0.
	\end{cases}
	\end{align}
	By the first equation in \eqref{aq}, we have $Ar_1=-B^*r_2$. 
	So $r_1=-A^{-1}B^*r_2$. 
	By substituting $r_1$ in  the second equation in \eqref{aq}, we get $BA^{-1}B^*r_2+Cr_2=0$. 
	Hence  $r_2^*BA^{-1}B^*r_2+r_2^*Cr_2=0$. 
	Therefore 
	\begin{align*}
	r_2^*B(A^{-*}+A^{-1})B^*r_2+r_2^*(C+C^*)r_2=0.
	\end{align*}
	Since $A$ and $C$ are  PD and PSD, respectively, we conclude that $B^*r_2=0$  and $(C+C^*)r_2=0.$
	By the first equation in \eqref{aq}, we have $Ar_1=0$ and so $r_1=0$. 
	Substituting $r_1=0$ in  the second equation in \eqref{aq}, gives $Cr_2=0$. 
	Therefore, 
	$C^*r_2=-Cr_2=0$ and so $(C_P^*-C_S) r_2=0$. Hence 
	$C_P^*r_2=C_S r_2$.	
	By the definition of $r$, we have $\mathfrak{P}_{EPSS}r=\mathfrak{P}_{EPSS} [0;r_2]=\mathcal{A}q$. 
	Therefore, 
	\begin{align}
	\label{mq}
	\begin{cases}
	\left((P_\alpha +A_P) P_\alpha^{-1} B_S^*+B^*_PP_\beta^{-1}(P_\beta+C_S)\right)r_2=Aq_1+B^*q_2,\\
	(P_\beta +C+C_P P_\beta^{-1}C_S-B_PP_\alpha ^{-1} B_S^*)r_2=-Bq_1+Cq_2.
	\end{cases}
	\end{align}
	By multiplying $r_2^*$ to the second equation in \eqref{mq} and using $C_S r_2=C_P^*r_2$, we conclude  
	\begin{align*}
	r_2^*(P_\beta +C+C_P P_\beta^{-1}C_P^*-B_PP_\alpha ^{-1} B_S^*)r_2=-r_2^*Bq_1+r_2^*Cq_2.
	\end{align*}
	Since $B^*r_2=0$ and  $r_2^*C=(C^*r_2)^*=0$,  we obtain 
	$$r_2^*(P_\beta+C_P P_\beta^{-1}C_P^* -B_PP_\alpha ^{-1} B_S^*)r_2=0.$$ 
	By using $B_S=B-B_P$, we have 
	$$r_2^*(P_\beta+C_P P_\beta^{-1}C_P^* +B_PP_\alpha ^{-1} B_P^*-B_PP_\alpha ^{-1} B^*)r_2=0,$$ 
	which gives   
	$$r_2^*(P_\beta+C_P P_\beta^{-1}C_P^* +B_PP_\alpha ^{-1} B_P^*)r_2=0.$$ 
	Since $P_\alpha$ and $P_\beta$ are HPD, we obtain that  $r_2=0$. Hence $r=0$ and so $null\left((\mathfrak{P}_{EPSS}^{-1}\mathcal{A})^2\right)= null(\mathfrak{P}_{EPSS}^{-1}\mathcal{A})$. Therefore,  $\text{index}(I-\Gamma_{EPSS} ) = 1$. 
	
	Now we prove $\nu(\Gamma_{\EPSS})<1$.
	Let $ 1 \neq \lambda \in \sigma(\Gamma_{\EPSS})$ with  $|\lambda|=1$.
	Similar to Theorem 3.2 in \cite{mas2019},  we know that there exists $0\neq r \in null (C+C^*)$ such that 
	\begin{align*}
	r^*(P_\beta-B_S P_\alpha^{-1} B_P^*+C_SP_\beta^{-1} C_P)r =\frac{1+\lambda}{1-\lambda}r^*Cr,
	\end{align*}
	which is contradiction with \eqref{condi2}.
\end{proof}
\begin{corollary}
	\label{semiconvcor}
	Suppose that $A$ and $C$ are  PD and PSD, respectively. Then, if $ null (C+C^*)  \subseteq null (C)$ and one of the following conditions holds true, then EPSS method is   semi-convergent.  
	\begin{enumerate}
		\item
		\label{part1}
		If for all {$0\neq r\in  null(C+C^*)$}, we have
		\begin{align}
		\label{condi3}
		r^*(P_\beta+C_SP_\beta^{-1} C_S^*)r\neq r^*(B_S P_\alpha^{-1} B_P^*)r.
		\end{align}
		\item
		\label{part2}
		{$0\neq r\in  null(C+C^*)$} implies  $r^*(B_S P_\alpha^{-1} B_P^*)r \leqslant  0$.
		\item
		\label{part3}
		{If $null(C+C^*) \subseteq null(B_S^*) \cup null(B_P^*)$}.
		\item
		\label{part4}
		If $C$ is PD or one of the matrices $B_S$ or $B_P$ is equal to 0.
	\end{enumerate}
\end{corollary}	



Let the matrix  $A$  and $C$ be  PD and PSD, respectively.
Moreover, suppose that   $P_\alpha$ and $P_\beta$ are two arbitrary HPD matrices.
If $null (C+C^*) \subseteq null (C)$ and one of the matrices $B_S$ or $B_P$ is equal to 0, by Corollary \ref{semiconvcor},   we deduce that $\nu(\Gamma_{\EPSS})<1$.
In this case, we present some special cases of the $\mathfrak{P}_{\EPSS}$ preconditioner.
\begin{enumerate}
	\item
	In \eqref{orgP}, let   $B_P=0$ and $C$ is HPSD.
	Therefore,    $\mathfrak{P}_{\EPSS}$  preconditioner   turns into the following preconditioner which is is an extension of the  HSS \cite{benzi262004} and PSS(DPSS) \cite{pan1722006}:
	\begin{equation}
	\label{mainHSS}
	\mathfrak{P}_{\EPSS}=  
	\begin{bmatrix}{(P_\alpha +A_P)P_\alpha^{-1} }&{0}\\{0}&{(P_\beta +C_P)P_\beta^{-1} }  \end{bmatrix}
	\begin{bmatrix}{P_\alpha +A_S  }&{B^*}\\{-B}&{P_\beta+C_S }\end{bmatrix}.
	\end{equation}
	\begin{enumerate}
		\item
		Let $C_S=0$ and $A_P=\frac{1}{2}(A+A^*)$. 
		\begin{enumerate}
			\item  
			Suppose that $P_\alpha =\alpha I$.
			If $P_\beta=\alpha I$,  then the preconditioner  \eqref{mainHSS} coincides with the HSS
			\cite{benzi262004,simoncini262004,bai892010}  which is denoted by   $\mathcal{P}_{\rm HSS}$ and if  $P_\beta=\beta I$, then the preconditioner  \eqref{mainHSS} coincides with the generalized HSS which we  denote  by  $\mathcal{P}_{\rm GHSS}$.
			\item 
			Let $P_\alpha =\alpha Q_1$ and  $P_\beta=\beta Q_2$,  where $\alpha>0$, $\beta>0$ and matrices $Q_1$ and $Q_2$ are HPD matrices. 
			Then the preconditioner  \eqref{mainHSS} coincides with the EHSS  which we denote by   $\mathcal{P}_{\rm EHSS}$.
			\item 
			Let $C=0$, $P_\alpha =\alpha I$ and $P_\beta =\alpha I+  Q $, where $Q$ is an HPD? matrix. 
			Then the preconditioner  \eqref{mainHSS} coincides with the regularized   HSS  (RHSS)  \cite{bai572017, chao762018} which is denoted by   $\mathcal{P}_{\rm RHSS}$.
		\end{enumerate}
		\item
		Let $C_S=0$   and  $A_P=A$.
		\begin{enumerate}
			\item  
			Suppose that $P_\alpha =\alpha I$ 
			and  $P_\beta=\alpha I$.   Then the preconditioner  \eqref{mainHSS} coincides with the PSS (or DPSS	\cite{pan1722006,shen2372014})
			which is denoted by   $\mathcal{P}_{\rm PSS}$ (or $\mathcal{P}_{\rm DPSS}$) and if  $P_\beta=\beta I$, then the preconditioner  \eqref{mainHSS} coincides with the generalized PSS \cite{wang2982016,huang752017} which we  denote  by  $\mathcal{P}_{\rm GPSS}$.
			\item 
			If $P_\alpha =\alpha Q_1$ and  $P_\beta=\beta Q_2$,  where  matrices $Q_1$ and $Q_2$ are HPD matrices, 
			then the preconditioner  \eqref{mainHSS} coincides with the EPSS  which we denote by   $\mathcal{P}_{\rm EPSS}$.
			A special case of this preconditioner is PDPSS preconditioner in  \cite{liang782018} where  $C=0$ and $B$  is a rank deficient matrix.
			\item 
			Let $C=0$, $P_\alpha =\alpha I$ and  $P_\beta =\alpha I+ Q $, where $Q$ is an HPD? matrix. 
			Then the preconditioner  \eqref{mainHSS} coincides with the regularized   PSS  (RPSS)  \cite{cao842018}  which is denoted by   $\mathcal{P}_{\rm RHSS}$.
			\item 
			Let $C=0$, $A_S=0$ and   $P_\alpha =\alpha A$. If  $P_\beta =\alpha Q$ then the preconditioner  \eqref{mainHSS} reduces to the 	PHSS  preconditioner in   \cite{bai982004} and if    $P_\beta =\beta Q$    then the preconditioner  \eqref{mainHSS} coincides with the 	AHSS  preconditioner in     \cite{bai272007}.
		\end{enumerate}
	\end{enumerate}
	\item
	In \eqref{orgP}, let  $B_S=0$ and $C$ be HPSD. So, the  ${\EPSS}$ preconditioner  turns into an extension of the  shift splitting preconditioner as
	\begin{equation}
	\label{mainSS}
	\mathfrak{P}_{\EPSS}=  
	\begin{bmatrix}{P_\alpha +A_P  }&{B^*}\\{-B}&{P_\beta+C_P }  \end{bmatrix}
	\begin{bmatrix}{P_\alpha^{-1}(P_\alpha +A_S) }&{0}\\{0}&{P_\beta^{-1}(P_\beta +C_S) }\end{bmatrix}.
	\end{equation}
	Let $C_S=0$ and $A_P=A$.
	Suppose that $P_\alpha =\alpha I$.
	If $P_\beta=\alpha I$,  then the preconditioner  \eqref{mainSS} coincides with the shift splitting  (SS) preconditioner
	which is denoted by   $\mathcal{P}_{\rm SS}$ and if  $P_\beta=\beta I$, then the preconditioner  \eqref{mainSS} reduces to the generalized  SS preconditioner  which we  denote  by  $\mathcal{P}_{\rm GSS}$. 
	Moreover, if $P_\alpha =\alpha Q_1$ and  $P_\beta=\beta Q_2$,  where  matrices $Q_1$ and $Q_2$ are HPD matrices, then the preconditioner  \eqref{mainSS} coincides with the ESS  preconditioner which we denote by   $\mathcal{P}_{\rm ESS}$.
	In  \cite{zheng3132017}, it has been shown that   if $AP_{\alpha}=P_{\alpha}A$, then the spectral radius of  ESS  iterative method is less than 1 but this condition is not necessary. 
	\begin{enumerate}
		\item  When $A$ is  HPD,  
		the  $\mathcal{P}_{\rm SS}$ preconditioner  was studied in 
		\cite{cao2722014} 
		and  $\mathcal{P}_{\rm GSS}$ preconditioner is studied in \cite{ren3112017,chen2692015,salkuyeh482015}. 
		Moreover the  $\mathcal{P}_{\rm ESS}$  preconditioner was presented in \cite{zheng3132017}.
		\item 
		When $A$ is  PD, the $\mathcal{P}_{\rm GSS}$ preconditioner  was studied in 
		\cite{cao712016,  cao492015, shi2692015, shen722016,salkuyeh482015}  
		and $\mathcal{P}_{\rm ESS}$  preconditioner  was studied in \cite{salkuyeh742017}.
	\end{enumerate}
\end{enumerate}
Therefore, by Corollary  \ref{semiconvcor} and Theorem 3.2 in \cite{mas2019},  these preconditioners  can be applied to singular  generalized saddle point problems \eqref{saddle}.

We use the EPSS preconditioner to accelerate the convergence of the  Krylov subspace methods such as the GMRES method to solving the system \eqref{saddle}.
At each step of applying the EPSS preconditioner $\mathfrak{P}_{\EPSS}$  within a Krylov subspace method, it is required to solve
the system of linear equations of the form
\begin{align}
\label{solveepss}
\mathfrak{P}_{\EPSS}
\begin{bmatrix}{y_1}\\{y_2}\end{bmatrix}=
\begin{bmatrix}{x_1}\\{x_2}\end{bmatrix},
\end{align}
where $x= \left[{x_1};{x_2}\right]$ is a given vector. There are different ways to solve this system. In the sequel we consider a special case of the EPSS (SEPSS) preconditioner and present its implementation. 

Let $C=D_C +L_C+U_C$, where the matrices $D_C$, $L_C$ and $U_C$ are  the block diagonal, strictly block lower triangular  and strictly
block upper triangular parts of the matrix A, respectively.   
We set 
\[
A_P=A, \quad C_P= L_C+U_C^*+D_C\quad {\rm and} \quad B_P=B.
\]
In this case, we have $A_S=0$,  $B_S=0$ and $C_P$ is block lower triangular PD matrix  and the preconditioner \eqref{orgP} takes a special preconditioner (SEPSS) as  the following form
\begin{equation}\label{EqFac}
\mathfrak{P}_{\SEPSS}= 
\begin{bmatrix}{A+P_\alpha}&{B^*}\\{-B}&{C_P+P_\beta} \end{bmatrix}
\begin{bmatrix}{I}&{0}\\{0}&{P_\beta^{-1}(C_S+P_{\beta})}\end{bmatrix}. 
\end{equation}

Let  $\Gamma_{\SEPSS}=I-2\mathfrak{P}_{\SEPSS}^{-1}\mathcal{A}$ be  the iteration matrix of the SEPSS iteration method. Using Theorem \eqref{semiconvth} and  Corollary \ref{semiconvcor}, the semi-convergence of the SEPSS method can be deduced.
In the implementation of this preconditioner within  a Krylov subspace method like GMRES  we need solving  systems of the form $\mathfrak P_{\SEPSS} [{y_1};{y_2}] =[{x_1};{x_2}]$. Using the factorization \eqref{EqFac}, the following algorithm can be written for the implementation of the SEPSS method.
\begin{algorithm}	
	\caption	{Solution $\mathfrak P_{\SEPSS} [{y_1};{y_2}] =[{x_1};{x_2}]$.}
	\label{algoritmSEPSS2}
	\begin{enumerate}
		\item{ Solve $(C_P+P_\beta )s_2=x_2$.}\\[-0.75cm]
		\item{ Set $s_1=x_1- B^*s_2$.}\\[-0.75cm]
		\item{ Solve $N y_1=(A+P_\alpha + B^*(C_P+P_\beta)^{-1}B)z_1=s_1$.}\\[-0.75cm]
		\item{ Solve $(C_P+P_\beta )z_2=B y_1+x_2$.}\\[-0.75cm]
		\item{ Solve $(C_S+P_\beta )y_2=P_\beta z_2$.} \\[-0.75cm]
	\end{enumerate}
\end{algorithm}

In Algorithm \ref{algoritmSEPSS2}, four  sub-systems with the coefficient matrices $C_P+P_\beta$, $C_S+P_\beta$ and $N$ should be solved.
If the matrix $P_\beta$ is assumed to be diagonal, then the matrix $C_P+P_\beta $  will be lower triangular and solving the corresponding system  can be accomplished by the forward  substitution.


Let $P_\alpha=\alpha Q_1$ and  $P_\beta=\beta Q_2$ where $Q_1$ and $Q_2$ are HPD matrices and independent of the parameters $\alpha$ and $\beta$. It is suggested to use a small value of $\alpha$ and one of the following values for $\beta$ in the EPSS method (see \cite{mas2019}): 
\begin{equation}
\label{alphbet}
\beta_* =\sqrt{\frac{\norm{\left(BA_\alpha^{-1}B^*+C_P\right)Q_2^{-1} C_{S} }_F}{\norm{Q_2}_F}}, ~~
\beta_{**}= \sqrt[4]{\frac{-\trace\left((BB^*+C_P^*C_P)(Q_2^{-1}C_S^2Q_2^{-1})\right)}{\trace(Q_2^2)}}.
\end{equation}


\section{Numerical examples}
\label{secexamp}

In this section, some numerical experiments  are   presented  to illustrate the effectiveness of the
SEPSS preconditioner for the saddle point problem \eqref{saddle} and the numerical results are compared with those of the  HSS, GHSS, EHSS, PSS, GPSS, and the EPSS preconditioners. 
The shift matrices in the EHSS, EPSS and the SEPSS preconditioners are chosen as following
\begin{equation}\label{PaPb}
P_\alpha=\alpha~\text{diag}(A+A^*)\quad \textrm{and} \quad  P_\beta =\beta \left(10^{-4}I+\text{diag}(C+C^*)\right),
\end{equation}
where $I$ is the identity matrix. Since, in general the matrix $\text{diag}(C+C^*)$ is HPSD, the term $10^{-4}I$ guarantees that the matrix $P_{\beta}$ is HPD. Since the $(1,1)$-block in all the test problems is PD matrix, we set $A_P=A$ in the PSS, GPSS, and EPSS preconditioners.
For the EPSS preconditioner, we set $C_P=C$. All runs are performed in \textsc{Matlab} 2017a on an Intel core(TM) i7-8550U (1.8 GHz) 16G RAM Windows 10 system.

\begin{table}
	\begin{center}
		\caption{Generic properties of the test  matrices $A$, $B$ and $C$ for different problems. }
		\label{Tablesize1}
		\scriptsize
		\begin{tabu}{|c|c|cccc|c|c|} 
			\hline 
			Grids	&  $n  $&   $ m$   &$nnz(A)$&$ nnz(B)$& $nnz(C)$
			\\ 
			\hline 
			$16\times16$ & 578& 256& 3826& 1800& 768
			\\  $32\times32$ & 2178& 1024& 16818& 7688& 3072
			\\  $64\times64$ & 8450& 4096& 70450& 31752& 12288
			\\  $128\times128$ & 33282& 16384& 288306& 129032& 49152
			\\ 	$256\times256$ & 132098& 65536& 1166386& 520200& 196608
			\\	\hline 
		\end{tabu}
	\end{center}
\end{table}
All the preconditioners are used to accelerate the convergence of the restarted GMRES method  with $restart=20$.
In our implementations, the initial guess is chosen to be a zero vector and the iteration is terminated once  the 2-norm of the  system residual is reduced by a factor of $10^9$.
The maximum number of the iterations and the maximum elapsed CPU time are set   $k_{\max} = 1000$  and  $t_{\max}=1000s$, respectively. 
Numerical results are presented in the tables. 
In the tables, ``IT" and ``CPU" stand for the iteration counts and  CPU time, respectively, and
\[
R_k=\frac{\|b-Au_k\|_2}{\|b\|} \quad {\rm and}  \quad E_k=\frac{\|e-u_k\|_2}{\|e\|_2},
\]
where $u_k$ is the computed solution. In the implementation of the preconditioners the LU factorization (resp., the Cholesky factorization in the HPD case) of the coefficient matrices in combination with the approximate minimum degree  reordering (AMD) (resp., symmetric AMD (SYMAMD) in the HPD case)  are used  for solving the sub-systems.
In the  example, the right-hand side vector  $b$   is set to be $b=\mathcal{A}e$,  where  $e= [1, 1,\ldots, 1]^T$.

We consider the Oseen problem
\begin{equation}
\label{oseen}
\begin{cases}
-\nu\Delta \mathbf {u}+ \mathbf w\cdot \nabla  \mathbf u+ \nabla p=\mathbf f, ~~~~ \text{in}~~~ \Omega, \\
~\hspace{2.3cm} \nabla \cdot   \mathbf u=0, ~~~~ \text{in}~~~ \Omega,
\end{cases}
\end{equation}
with suitable boundary conditions on $ \partial\Omega $, where $\Omega\subset \mathbb{R}^2$  is a bounded domain and $ \mathbf w $ is a given divergence free field. The	parameter $ \nu > 0 $ is the viscosity, the vector field $\mathbf u $ stands for the velocity and $p $ represents the pressure.
Here $\Delta$, $\nabla\cdot$ and $\nabla$ stand for the Laplace operator, the divergence operators and the gradient, respectively.
The Oseen problem \eqref{oseen} is obtained from the linearization  of the steady-state Navier-Stokes equation by the Picard iteration where the vector field $\mathbf w$ is the approximation of $\mathbf u$ from the previous Picard iteration.
It is known that many discretization schemes for \eqref{oseen} will lead to a generalized saddle point  problems of the form \eqref{saddle}.
We use the  {stabilized (the Stokes stabilization)} $Q_1-P_0$ finite element method for the leaky lid driven cavity problems on uniform-grids on the unit square, with $\nu=0.01$.	In these cases, the matrix  $A$ is non-symmetric, but it  is positive definite.	


We use the IFISS software package developed by Elman et al.  \cite{elman2007} to generate the linear systems corresponding to  $16 \times 16$, $32 \times 32$, $64\times64$,  $128\times128$ and $256\times256$ grids.
It is noted that  the matrix $A$ is PD,  { $C\neq 0$} and $null(B^*)\cap null(C)\neq 0$.
The generic properties of the test problems  are given in  Table \ref{Tablesize1}.
To generate the preconditioners the parameters $\alpha$ and $\beta$ are set $\alpha=10^{t_\alpha}$  and $\beta=10^{t_\beta}$ where $t_\alpha=t_\beta=-4:0.25:4$ (in \textsc{Matlab} notation).
Numerical results are presented in Table \ref{navstolidq1p01}.
For each grid, the results of a pair of $(\alpha,\beta)$ with minimum number of iterations is reported.
When the number of iterations for some pairs of $(\alpha,\beta)$ are the same, then the minimum CPU time is reported.
In the table, we use   SEPSS$_*$    for  SEPSS  method  with $(\alpha , \beta)=(10^{-4},\beta_*)$ and  SEPSS$_{**}$  for  SEPSS method  with $(\alpha , \beta)=(10^{-4}, \beta_{**})$.

The reported numerical results show that, the SEPSS preconditioner  outperforms the other  methods from both the iteration counts and CPU time point of view.
As the numerical results show  the strategies presented in   \eqref{alphbet} are quite suitable to estimate the optimal values of $\alpha$ and $\beta$.

\begin{table}
	\begin{center}
		\caption{Numerical results for the Oseen problem. }
		\label{navstolidq1p01}
		\scriptsize
		\scalebox{.7}
		{
			\begin{tabu}{|c|c|c|c|c|c|c|c|c|c|c|c|c|c|c|}
				\hline
				&grid &prec.    \navierOTHER{GSS}{GHSS}{GPSS}{ESS}{EHSS}{EPSS}{SS}{HSS}{PSS}\navierSEPSS{SEPSS}{SEPSS$_{*}$}{SEPSS$_{**}$} \\ \hline 
				\multirow[c]{20}{*}{\rotatebox{90}{?\large\bfseries  GMRES}}&\multirow[c]{5}{*}{\bfseries $16\times 16 $}& $t_\alpha$   \navierOTHER{-3.75}{-2.00}{-1.75}{-4.00}{-0.75}{-4.00}{-3.75}{-1.75}{-1.75}\navierSEPSS{-4.00}{-4.00}{-4.00} \\ 
				&&$t_\beta$   \navierOTHER{-3.50}{-1.50}{-1.75}{-3.25}{-0.25}{3.00}{-3.75}{-1.75}{-1.75}\navierSEPSS{0.25}{0.39}{-0.01} \\ 
				&&IT   \navierOTHER{4}{27}{37}{3}{28}{12}{4}{32}{37}\navierSEPSS{10}{11}{10} \\ 
				&&CPU   \navierOTHER{0.00}{0.03}{0.02}{0.00}{0.03}{0.01}{0.00}{0.01}{0.02}\navierSEPSS{0.00}{0.03}{0.02} \\ 
				&&$R_K$   \navierOTHER{4.8e-10}{9.1e-10}{6.6e-10}{2.9e-12}{9.2e-10}{6.3e-10}{2.6e-10}{6.1e-10}{6.6e-10}\navierSEPSS{7.3e-10}{6.8e-10}{3.0e-10} \\ 
				\clnavier
				&\multirow[c]{5}{*}{\bfseries $32\times 32 $}& $t_\alpha$   \navierOTHER{-4.00}{-2.25}{-2.00}{-4.00}{-0.75}{-4.00}{-4.00}{-2.25}{-2.25}\navierSEPSS{-4.00}{-4.00}{-4.00} \\ 
				&&$t_\beta$   \navierOTHER{-4.00}{-2.25}{-2.50}{-4.00}{-0.75}{2.75}{-4.00}{-2.25}{-2.25}\navierSEPSS{0.00}{0.41}{0.14} \\ 
				&&IT   \navierOTHER{5}{41}{46}{3}{41}{19}{5}{41}{48}\navierSEPSS{10}{12}{11} \\ 
				&&CPU   \navierOTHER{0.03}{0.12}{0.10}{0.02}{0.13}{0.04}{0.03}{0.14}{0.05}\navierSEPSS{0.02}{0.04}{0.03} \\ 
				&&$R_K$   \navierOTHER{1.0e-10}{9.7e-10}{9.8e-10}{2.2e-12}{8.7e-10}{1.0e-09}{1.0e-10}{9.7e-10}{8.5e-10}\navierSEPSS{9.6e-10}{8.8e-10}{3.4e-10} \\ 
				\clnavier
				&\multirow[c]{5}{*}{\bfseries $64\times 64 $}& $t_\alpha$   \navierOTHER{-4.00}{-2.50}{-2.50}{-4.00}{-1.25}{-4.00}{-4.00}{-2.50}{-2.50}\navierSEPSS{-4.00}{-4.00}{-4.00} \\ 
				&&$t_\beta$   \navierOTHER{-3.75}{-3.00}{-3.00}{-4.00}{-0.50}{3.00}{-4.00}{-2.50}{-2.50}\navierSEPSS{0.00}{0.41}{0.29} \\ 
				&&IT   \navierOTHER{7}{65}{61}{3}{67}{43}{7}{73}{71}\navierSEPSS{11}{14}{12} \\ 
				&&CPU   \navierOTHER{0.23}{1.32}{0.36}{0.15}{1.92}{0.24}{0.19}{1.34}{0.41}\navierSEPSS{0.14}{0.13}{0.12} \\ 
				&&$R_K$   \navierOTHER{7.1e-10}{8.8e-10}{9.7e-10}{1.1e-11}{8.9e-10}{9.9e-10}{3.4e-10}{9.1e-10}{9.3e-10}\navierSEPSS{9.8e-10}{3.3e-10}{8.5e-10} \\ 
				\clnavier
				&\multirow[c]{5}{*}{\bfseries $128\times 128 $}& $t_\alpha$   \navierOTHER{-4.00}{-2.75}{-2.75}{-4.00}{-1.25}{-4.00}{-4.00}{-2.75}{-3.00}\navierSEPSS{-4.00}{-4.00}{-4.00} \\ 
				&&$t_\beta$   \navierOTHER{-4.00}{-3.75}{-3.75}{-3.50}{-1.25}{2.75}{-4.00}{-2.75}{-3.00}\navierSEPSS{0.25}{0.41}{0.41} \\ 
				&&IT   \navierOTHER{12}{104}{95}{3}{103}{87}{12}{144}{114}\navierSEPSS{13}{15}{15} \\ 
				&&CPU   \navierOTHER{2.01}{12.33}{3.95}{1.05}{13.73}{2.38}{2.33}{14.72}{4.20}\navierSEPSS{0.81}{0.96}{0.86} \\ 
				&&$R_K$   \navierOTHER{6.4e-10}{9.6e-10}{9.5e-10}{1.9e-10}{1.0e-09}{9.9e-10}{6.4e-10}{9.9e-10}{9.3e-10}\navierSEPSS{9.3e-10}{6.4e-10}{6.2e-10} \\ 
				\clnavier
				\multirow[c]{20}{*}{\rotatebox{90}{?\large\bfseries  FGMRES}}&\multirow[c]{5}{*}{\rotatebox{0}{?\bfseries $16\times 16 $}}& $t_\alpha$   \navierOTHER{$-4.00$}{$-2.00$}{$-1.75$}{$-3.00$}{$-0.75$}{$-0.50$}{$-3.00$}{$-1.75$}{$-1.75$}\navierSEPSS{$-3.75$}{$-4.00$}{$-4.00$} \\ 
				&&$t_\beta$   \navierOTHER{$-2.25$}{$-1.25$}{$-1.75$}{$-1.00$}{$-0.25$}{$-0.75$}{$-3.00$}{$-1.75$}{$-1.75$}\navierSEPSS{$0.25$}{$0.39$}{$-0.01$} \\ 
				&&IT   \navierOTHER{$4$}{$20$}{$26$}{$4$}{$20$}{$28$}{$6$}{$23$}{$26$}\navierSEPSS{$10$}{$11$}{$10$} \\ 
				&&CPU   \navierOTHER{$0.16$}{$0.07$}{$0.04$}{$0.15$}{$0.08$}{$0.09$}{$0.33$}{$0.07$}{$0.10$}\navierSEPSS{$0.03$}{$0.09$}{$0.09$} \\ 
				&&$R_K$   \navierOTHER{$7.9e-08$}{$6.9e-08$}{$9.0e-08$}{$8.0e-08$}{$7.7e-08$}{$8.5e-08$}{$1.4e-08$}{$7.6e-08$}{$9.0e-08$}\navierSEPSS{$3.9e-08$}{$4.2e-08$}{$1.9e-08$} \\ 
				\clnavier
				&\multirow[c]{5}{*}{\rotatebox{0}{?\bfseries $32\times 32 $}}& $t_\alpha$   \navierOTHER{$-4.00$}{$-2.00$}{$-2.25$}{$-3.75$}{$-0.75$}{$-0.75$}{$-3.50$}{$-2.25$}{$-2.00$}\navierSEPSS{$-3.75$}{$-4.00$}{$-4.00$} \\ 
				&&$t_\beta$   \navierOTHER{$-3.00$}{$-2.50$}{$-1.75$}{$-0.50$}{$-0.75$}{$-0.25$}{$-3.50$}{$-2.25$}{$-2.00$}\navierSEPSS{$0.00$}{$0.41$}{$0.14$} \\ 
				&&IT   \navierOTHER{$5$}{$29$}{$38$}{$5$}{$28$}{$38$}{$7$}{$30$}{$39$}\navierSEPSS{$10$}{$12$}{$10$} \\ 
				&&CPU   \navierOTHER{$0.65$}{$0.71$}{$0.17$}{$0.39$}{$0.56$}{$0.33$}{$0.74$}{$0.99$}{$0.15$}\navierSEPSS{$0.20$}{$0.26$}{$0.24$} \\ 
				&&$R_K$   \navierOTHER{$8.7e-08$}{$9.3e-08$}{$8.4e-08$}{$9.2e-08$}{$8.8e-08$}{$8.7e-08$}{$8.6e-08$}{$9.4e-08$}{$7.6e-08$}\navierSEPSS{$4.5e-08$}{$5.8e-08$}{$5.8e-08$} \\ 
				\clnavier
				&\multirow[c]{5}{*}{\rotatebox{0}{?\bfseries $64\times 64 $}}& $t_\alpha$   \navierOTHER{$-4.00$}{$-2.25$}{$-2.50$}{$-4.00$}{$-1.00$}{$-1.25$}{$-3.75$}{$-2.50$}{$-2.25$}\navierSEPSS{$-3.50$}{$-4.00$}{$-4.00$} \\ 
				&&$t_\beta$   \navierOTHER{$-2.75$}{$-3.25$}{$-2.00$}{$-0.50$}{$-0.75$}{$0.50$}{$-3.75$}{$-2.50$}{$-2.25$}\navierSEPSS{$0.00$}{$0.41$}{$0.29$} \\ 
				&&IT   \navierOTHER{$8$}{$43$}{$59$}{$6$}{$42$}{$60$}{$9$}{$46$}{$69$}\navierSEPSS{$11$}{$13$}{$12$} \\ 
				&&CPU   \navierOTHER{$5.59$}{$7.43$}{$2.39$}{$4.31$}{$7.25$}{$1.68$}{$6.55$}{$8.13$}{$2.17$}\navierSEPSS{$3.44$}{$4.82$}{$4.23$} \\ 
				&&$R_K$   \navierOTHER{$4.0e-08$}{$9.6e-08$}{$9.9e-08$}{$9.9e-08$}{$9.8e-08$}{$9.6e-08$}{$8.3e-08$}{$9.7e-08$}{$9.7e-08$}\navierSEPSS{$7.7e-08$}{$4.8e-08$}{$4.6e-08$} \\ 
				\clnavier
				&\multirow[c]{5}{*}{\rotatebox{0}{?\bfseries $128\times 128 $}}& $t_\alpha$   \navierOTHER{$-4.00$}{$-2.75$}{$-2.75$}{$-4.00$}{$-1.25$}{$-1.50$}{$-3.75$}{$-2.75$}{$-2.50$}\navierSEPSS{$-3.75$}{$-4.00$}{$-4.00$} \\ 
				&&$t_\beta$   \navierOTHER{$-3.25$}{$-3.25$}{$-2.25$}{$-0.00$}{$-0.75$}{$0.50$}{$-3.75$}{$-2.75$}{$-2.50$}\navierSEPSS{$0.25$}{$0.42$}{$0.41$} \\ 
				&&IT   \navierOTHER{$13$}{$71$}{$106$}{$10$}{$71$}{$108$}{$15$}{$80$}{$139$}\navierSEPSS{$13$}{$14$}{$14$} \\ 
				&&CPU   \navierOTHER{$27.57$}{$47.39$}{$13.43$}{$20.43$}{$43.25$}{$12.38$}{$32.05$}{$46.00$}{$14.69$}\navierSEPSS{$12.38$}{$16.20$}{$14.43$} \\ 
				&&$R_K$   \navierOTHER{$8.3e-08$}{$9.7e-08$}{$9.4e-08$}{$5.9e-08$}{$9.4e-08$}{$9.6e-08$}{$9.2e-08$}{$9.5e-08$}{$9.7e-08$}\navierSEPSS{$6.5e-08$}{$5.0e-08$}{$4.5e-08$} \\ 
				\hline
			\end{tabu}
		}
	\end{center}
\end{table}

\section{Conclusion}

We have investigated the semi-convergence of the extended PSS (EPSS) method for singular saddle pint problems. Then we have applied the a special case of the EPSS preconditioner to accelerate the convergence of the Krylov subspace methods like GMRES. Numerical results showed that the proposed outperforms many existing methods.

\section*{Acknowledgment}
The work of Davod Khojatseh Salkuyeh is partially supported by University of Guilan.

\enddocument